\documentclass[amsart]{article}

\usepackage{latexsym}
\usepackage{amssymb}
\usepackage{amstext}
\usepackage{amscd}
\usepackage[dvips]{graphicx}
\usepackage{graphicx}
\usepackage{amsthm}
\usepackage{amsfonts}	
\usepackage{amsmath}

\newtheorem{teo}	{Theorem}[section]

\newtheorem*{teoa}	{Theorem A}
\newtheorem*{teob}	{Theorem B}

\newtheorem{conj}	{Conjecture}

\newtheorem{defi}	{Definition}

\newtheorem{pr}		{Proposition}[section]
\newtheorem{lm}		{Lemma}[section]
\newtheorem{cor}	{Corollary}[section]

\newtheorem{rem}	{Remark}

\newtheorem*{corb}{Corollary C}

\newcommand{\N}		{\mathbb{N}}
\newcommand{\Z}		{\mathbb{Z}}

\newcommand{\C}		{\mathbb{C}}
\newcommand{\Cb}	{\overline{\mathbb{C}}}
\newcommand{\Ctwo}	{{\mathbb{C}}^2}

\newcommand{\im}	{{\rm Im\, }}

\newcommand{\bpl}	{\bigl\{ }
\newcommand{\bpr}	{\bigr\} }

\newcommand{\ww}  [2]	{W^{#1}({#2})}
\newcommand{\www} [3]	{W^{#1}_{#3}({#2})}

\newcommand{\graph}	{{\rm graph\, }}
\renewcommand{\d}	{{\rm dist\, }}

\newcommand{\reg}	{\mathcal{R}}
\newcommand{\per}	{\mathcal{P}{\hspace{-2pt}}{\it er}}

\newcommand{\e}		{\varepsilon}

\newcommand{\supp}	{{\rm supp\, }}

\renewcommand{\mod}	{{\rm mod}}

\begin{document}

\title{\bf Equivalent conditions for hyperbolicity on
partially hyperbolic holomorphic map}

\author{Francisco Valenzuela-Henr\'{\i}quez\thanks{Thanks for the
financial support to CNPq-IMPA and PEC-PG and Universidad
Cat\'olica de Valparaiso.}\\
\texttt{\small{fvalenzh@impa.br,
pancho.valenzuela.math@gmail.com}}}

\maketitle

\begin{abstract}
Let $f:\C\sp n\rightarrow\C\sp n$, $n\geq2$, be a  biholomorphism  and
 let $\Lambda\subseteq \C\sp n$ be a compact $f$-invariant set  such
that $f|\Lambda$ is partially hyperbolic. We give equivalent
conditions to hyperbolicity on $\Lambda$. In the particular case of
generalized H\'enon map with dominated splitting in the Julia set $J$,
we characterize the hyperbolicity of $J$.
\end{abstract}

\tableofcontents

\section{Introduction}\label{sec:intro}

In the theory of complex dynamical systems, a well known seminal area
is the study of rational maps on the Riemann sphere. For complex
dynamics in several variables, the study of polynomial automorphisms
of $\C^2$ is the first step for {\it a global understanding of
holomorphics dynamic in higher dimension.}

% Moreover, this area is particularly interesting because of its
% connections to some fundamental questions of dynamical systems via
% two real dimensional dynamics and because of its connection to some
% powerful techniques via one dimensional complex tools.

One of the first results in this direction, were given by Friedland
and Milnor in \cite{fm}. They proved that for polynomial
automorphism in $\Ctwo$, the only systems (module conjugation by a
polynomial automorphism) that exhibit rich dynamics are the so
called generalized H\'enon maps (or by simplicity, complex H\'enon
maps).

Such a maps are obtained as a finite composition of
maps of the form $(y,p(y)-bx)$, where $p$ is a polynomial of degree at
least two, and $b \in \Bbb{C}^*$. Complex H\'{e}non maps have been a
subject of serious study, with foundational work done in the early
1990's by Hubbard \cite{h}, Hubbard and Oberste-Vorth
\cite{ho1,ho2}, Bedford and Smillie
\cite{bs1,bs2,bs3}, Bedford, Lyubich and Smillie \cite{bls1} and
Forn\ae{}ss and Sibony \cite{fs1}.

As in the case of rational maps, complex H\'enon map have well
defined a Julia set $J$ (see \cite{bs1}). This set is a
compact invariant set, and it contains the supports of the unique
measure of maximal entropy. Such measure exists due to the works of
Bedford and Smillie \cite{bs1}. Denote by $J^*$  the support of the
measure of maximal entropy.

A significant open question in the study of complex H\'{e}non maps is
whether $J = J^*$. Bedford and Smillie \cite{bs1} have shown that if
$f$ is uniformly hyperbolic on $J$, then $J=J^*$. Moreover,
Forn\ae{}ss proved that if $f$ is uniformly hyperbolic on $J^*$ and
$f$ is not volume preserving, then $J=J^*$ \cite{F}. In the setting
of complex H\'enon maps, hyperbolicity is the natural
generalization of the expansiveness on the Julia set for rational
maps on $\Cb$.

Motivated by the results above, in this work we establish equivalent
conditions to hyperbolicity for biholomorfisms of $\C\sp n$ with a
compact invariant set $\Lambda\subset\C\sp n$, under the hyphothesis
of partial hyperbolicity. It is well known that for a partial
hyperbolic map, we have strong stable manifolds and center-unstable
leaves \cite{HPS, shub}. So we can to characterize the uniform
hyperbolicity in terms of the behavior along the center-unstable
leaves.

More precisely, we say that the a local  $cu$-leaf
$\www{cu}{x}{\e}$ is {\it dynamically defined},
if for every $0<\e\ll1$, it is included in the local
unstable set of $x$. We say that the $cu$-leaves are {\it forward
expansive} if there exist a uniform constant $c>0$ such that for every
$x\in\Lambda$, and any $y\in\www{cu}{x}{\varepsilon}$, there exists
$n\in\N$, such that \[\d(f\sp n(y),f\sp n(x))> c.\]

The reader can find a more precise statement of the notions of
partial hyperbolicity, forward expansiveness and dynamically defined
in Section~\ref{preliminaries}. Our main theorem is the following.

\begin{teoa}\label{mainteoa} Let $f:\C\sp n\rightarrow\C\sp n$,
$n\geq 2$, be a biholomorphism which is partially hyperbolic on the
compact invariant set $\Lambda$. Then the following statements are
equivalent:
\begin{enumerate}
\item The function $f$ is uniformly hyperbolic on
$\Lambda$.
\item The $cu$-leaves are forward expansive.
\item The $cu$-leaves are dynamically defined.
\end{enumerate}
\end{teoa}

% At this point it is important to highlight some observations.

One of the motivations of the previous Theorem appears in the
study of complex H\'enon map with dominated splitting. In the setting
of dissipative H\'enon maps, do\-mi\-na\-ted splitting in $J$ and
partial hy\-per\-bo\-li\-ci\-ty are equivalents
(see Proposition \ref{1aprop}). It is important to note that both,
partially hyperbolic and dominated splitting are two ways to
relax hyperbolicity.

Theorem A establishes that it is enough assume forward expansiveness
or the dynamical definition of the $cu$-leaf to guarantee
hyperbolicity.

We must remark that for real maps with dominated splitting in
manifolds, the condition $cu$-leaf dynamically defined is not enough
to conclude hyperbolicity. In order to conclude hyperbolicity it is
required to add the hyphothesis of Kupka-Smalle over the
diffeomorphisms (see \cite{ps1,ps2} for surfaces and codimension one
context respectively).

Nevertheless motivated by the works \cite{ps1, ps2}, we conjecture
that:

\begin{conj}
Generically (Kupka-Smalle) complex H\'enon maps with dominated
splitting are hyperbolics.
\end{conj}

We study the relation between uniform hyperbolicity and forward
expansiveness  because in the H\'enon maps the saddle periodic points
(which are dense in $J^*$ \cite{bs1}) satisfy a non-uniform forward
expansivity condition. Moreover, for dissipative H\'enon maps with
dominated splitting, Proposition \ref{hol6} states that in every
center-unstable leaf there exist many points with (uniform) forward
expansivity.

Another motivation of Theorem A is the notion of {\it
quasi-expanding} due to Bedford and Smillie in \cite{bs8}. Roughly
speaking, quasi-expanding corresponds to an uniform forward
expansivity in the periodic saddle point respect to the complex
structure induced by the dynamics. The authors establish that a
(topological) expansive quasi-hyperbolic
map (quasi-expanding and quasi-contracting) is hyperbolic. Dissipative
complex H\'enon maps with dominated splitting are quasi-contracting.
We obtain the same conclusion of Bedford and Smillie just assuming
forward expansivity on the whole Julia set $J$.

% The main idea of the proof is to note that, when the center unstable
% leaves are dynamically defined, then they are holomorphics. In
% consequence, they are unique and the center-unstable direction  is
% in fact a strong unstable direction.

The sketch of the proof of the Theorem A, is essentially the
following: firstly we establish the equivalence between forward
expansivity and dynamically defined. Once the $cu$-leaves are
dynamically defined, then they are holomorphics. In
consequence, they are uniques and the center-unstable direction  is
in fact a strong unstable direction.

The existence of a $cu$-leaf, follows from a classical argument using
the graph transform operator (see Theorem \ref{h_p_thm}). It is
possible to define the graph transform operator, in an appropriated
(complete and metric) space of Lipschitz maps. In such case
this operator is a contraction and the $cu$-leaf is the unique fixed
point. The $cu$-leaf given by the graph transform operator is only
$C\sp1$. To prove the holomorphy of the leaf it is necessary to prove
that we can approximate the $cu$-leaf by holomorphic Lipschitz map
(iterate of a holomorphic Lipschitz map by the graph transform
operator). Hence, knowing that the convergence in the space of
Lipschitz function is the convergence uniform on compact part, we
conclude the holomorphy of the $cu$-leaf. The delicate step is
guarantee the explained above, only using the dynamically defined
property.

Among the dissipative H\'enon maps, it is possible to obtain a more
refined equivalence to the hyperbolicity of the Julia set. Follows
from \cite{bls1}, that

\[J\sp*=\overline{\bigcup\bpl \supp(\nu)\,:\,\nu\, \textrm{ is $f$-
invariant}\,
\textrm{hyperbolic}\, \bpr}.\]
So  we can  define the set
\[J_0=\overline{\bigcup\bpl \supp(\nu)\,:\,\nu\, \textrm{ is $f$-
invariant}\,
\textrm{and has a zero exponent}\, \bpr}.\]
Note that by definition, $J_0$ is a compact $f$-invariant set.

\begin{teob}\label{mainteob}
Let $f:\Ctwo\rightarrow\Ctwo$ be a dissipative complex H\'enon map,
with dominated spli\-tting in $J$. The following statement are
equivalents:
\begin{enumerate}
\item $J$ is uniformly hyperbolic,
\item $J_0=\emptyset$,
\item The set of periodic (saddle) points is uniformly
hyperbolic.
\item The set of periodic (saddle) points is uniformly expanding
at the period.
\end{enumerate}
\end{teob}

Statements 3 and 4 in the theorem follows from $J^*$ be an homoclinic
class. For a precise statement of uniform expansion at the period see
Definition~\ref{defi:7}. An immediate Corollary from the
Theorem~%\ref{mainteob}
B is the following result.

\begin{corb}
Let $f$ be a dissipative complex H\'enon map, with dominated
spli\-tting in $J$. Then $J$ is hyperbolic if, and only if, every
$f$-invariant measure supported in $J$ is hyperbolic.
\end{corb}

The paper is organized as follows. In Section \ref{preliminaries} we
state some results and tools related with partially hyperbolicity,
and we define the notions of forward expansivity and dynamically
defined. Also we state the existence of stable/center-unstable
manifold for partially hyperbolic systems. In Section \ref{section3},
we present the Theorem of existence of stable/center-unstable
manifolds in the holomorphic context. From Section \ref{section4}
until \ref{section:7}, we present the proof of the Theorem A. Finally
in Section \ref{section8}, we present some formalisms for complex
H\'enon maps and prove of Theorem B, and another.
\newline

\noindent
{\it Acknowledgements:}
This article is a part of my PhD Thesis under the direction of
Enrique Pujals at IMPA. I would like to thank him for his
guidance.

\section{Preliminaries}\label{preliminaries}

In this section we recall several classic results of dynamical
systems, and we write them in the context of complex and holomorphic
dynamics in $\C\sp n$ for any $n\geq2$.

We define the open polydisc of center $0$ and radio $r>0$ in $\C\sp
k$ as the set
\[\Delta_k(0,r)=\bpl z\in\C\sp k\, : \, |z_i|< r,\, \textrm{
for every }\, i=1,\ldots,k\bpr.\]

We recall the notion of partially hyperbolic (see for references
\cite{pesin} or \cite{hds1b}).

\begin{defi}\label{part_hip}
Let $f:\C\sp n\rightarrow\C\sp n$ be a biholomorphism and
$\Lambda\subset\C\sp n$ denote be a compact $f$-invariant  set. We
say that $f$ is partially hyperbolic (in the broad sense) on
$\Lambda$, if there exist a $Df$-invariant splitting $T_\Lambda\C\sp
n=E\oplus F$, and constant $0<\lambda<\mu$ and $C>0$  such that,
\begin{itemize}
\item[1.]
$||Df\sp n(x)|_{E(x)}||\leq C\lambda\sp n$ for all $n\geq0$,
\item[2.] $||Df\sp{-n}(x)|_F(x)||\leq C\mu\sp{-n}$ for
all $n\geq0$.
\end{itemize}

\end{defi}

Clearly, either $\lambda<1$ and/or $\mu>1$ and without lost of
generality in what follows we assume that $\lambda<1$. In this case,
the subspace $E(x)$ is stable and it will denoted by $E\sp s(x)$. Also
we denote by $l$ the complex dimension of the space $E(x)\sp s$ and by
$k$ the complex dimension of the space $F(x)$.

% RECALL THAT DOMINATED SPLTTING HENON MAP BE PARTIALLY HIPERBOLIC   %
% IN
% INTRODUCTION

Denote by $Emb\sp r(\Delta_k(0,1),\C\sp n)$ the set of
$C\sp r$-embeddings of $\Delta_k(0,1)$ on $\C\sp n$. Two point $x,\,
y\in \C\sp n$ are {\it forward $\rho$-asymptotic} under $f$, if
$d(f\sp n(x),f\sp n(y))\leq C\rho\sp n$ for all $n\geq 0$ and some
constant $C>0$. Similarly, we define {\it backward $\rho$-asymptotic}
as forward $\rho$-asymptotic for $f\sp{-1}$.

Recall by \cite{HPS} (see also \cite{shub}) that a partially
hyperbolic systems to admit the existence of stable/center-unstable
manifolds.

\begin{teo}\label{hol0}
Let $f$ be a biholomorphism in $\C\sp n$, such that $f$ is
partially hyperbolic on $\Lambda$ with splitting $T_\Lambda\C\sp
n=E\sp s\oplus F$. Then there exist two continuous functions
$\phi\sp{s}:\Lambda\rightarrow
Emb\sp\infty(\Delta_k(0,1),\C\sp n)$ and
$\phi\sp{cu}:\Lambda\rightarrow
Emb\sp1(\Delta_l(0,1),\C\sp n)$ such that, with
$\www{s}{x}{\e}=\phi\sp{s}(x)\Delta_k(0,\e)$ and
$\www{cu}{x}{\e}=\phi\sp{cu}(x)\Delta_l(0,\e)$, the fo\-llo\-wing
properties hold:
\begin{itemize}
\item[a)] $T_x\www{s}{x}{\e}=E\sp s(x)$ and
$T_x\www{cu}{x}{\e}=F(x)$,
\item[b)] for all $0<\e_1<1$ there exist $\e_2$ such that
\[f(\www{s}{x}{\e_2})\subset \www{s}{f(x)}{\e_1}\]and
\[f\sp{-1}(\www{cu}{x}{\e_2})\subset \www{cu}{f\sp{-1}(x)}{\e_1}.\]
\end{itemize}
Then sets $\www{s}{x}{1}$ with $x\in\Lambda$ are submanifolds
of $\C\sp n$, and are characterized as those points locally forward
$\rho$-asymptotic with $x$, for some $\lambda\leq\rho<\mu$.
\end{teo}

The sets $\www{s}{x}{1}$ are the local stable manifolds in the point
$x\in\Lambda$. We name the sets $\www{cu}{x}{\e}$, the center-unstable
leaf or $cu$-leaf. Clearly in the case $\mu>1$, the subspace $F(x)$ is
unstable and the $cu$-leaf are unstable manifolds.

In the holomorphic context we can say even more about the stable
manifold.

\begin{teo}\label{hol1}
Let $f$ be as in the Theorem \ref{hol0}. Then the local stable
manifolds $\bpl\www{s}{x}{1}\bpr_{x\in\Lambda}$ are holomorphic
submanifolds of $\C\sp n$.
\end{teo}

This Theorem is part of the folklore and we prove them in the
following section. In his proof is introduced an important technique,
that we use later in the proof of Proposition \ref{ws_holm}.

To end this section, we recall some basic definition. The {\it
unstable set} of a point $x$ for $f$, is the set
\[W\sp u(x)=\bpl y\in\C\sp n\, :\,
d(f\sp{-n}(x),f\sp{-n}(y))\rightarrow 0,\, \textrm{ when }\,
n\rightarrow \infty \bpr,\]
where $d$ is the Euclidean distance. Similarly, the {\it local
unstable set} of size $\e$ is the set
\[W\sp u_\e(x)=\bpl y\in W\sp u(x)\, :\,
d(f\sp{-n}(x),f\sp{-n}(y))\leq\e,\, \textrm{ for every }
n\geq0 \bpr.\]

It is know that for every $0<\e\leq1$ there exist $\delta>0$
such that for every $x\in\Lambda$,
$\www{u}{x}{\delta}\subseteq\www{cu}{x}{\e}$,
however in general the opposite inclusion not hold if we not have
good properties in the asymptotic behavior of $Df$.

We end this section, recalling the definition of $cu$-leaves
dynamically defined and the notion $cu$-forward expansivity for a
biholomorphism $f$ in $\C\sp n$, that is partially hyperbolic on
$\Lambda$.

\begin{defi}
We say that the $cu$-leaves are dynamically defined, if for
every\linebreak $0<\e\ll1$, $\www{cu}{x}{\e}\subset \www{u}{x}{loc}$
for all $x\in \Lambda$.
\end{defi}

\begin{defi}
We say that $f$ is forward expansive in the center-unstable leaves or
$cu$-forward expansive, if there exist a uniform constant $c>0$ such
that for\linebreak every $x\in\Lambda$, and any
$y\in\www{cu}{x}{\varepsilon}$, there exists $n\in\N$, such that
\[\d(f\sp n(y),f\sp n(x))> c.\] We say that the constant $c$ is the
expansiveness constant.
\end{defi}

\section{Holomorphic Hadamard-Perron Theorem}\label{section3}

A way to see the proof of the Theorem \ref{hol0}, is applying
the classical Hadamard-Perron Theorem. We will use the notation and
the ``technique'' of this Theorem, to prove many of the statement in
the following sections, and use the version of this theorem stated in
the book \cite{KH}. In this section we explain and present a sketch of
the proof of this Theorem and prove the Theorem \ref{hol1}.

\begin{teo}[{\bf Hadamard-Perron Theorem}]\label{h_p_thm} Let
$\lambda<\mu$, $r\geq1$ and for each $m\in\Z$ let
$f_m:\C^n\rightarrow\C^n$ be a $C^r$
diffeomorphisms such that for $(x,y)\in\C^l\oplus\C^{k}$,
\[f_m(x,y)=(A_mx+\alpha_m(x,y),B_my+\beta_m(x,y)),\]
for some linear maps $A_m:\C^l\rightarrow\C^l$ and
$B_m:\C^{k}\rightarrow\C^{k}$ with $||A_m^{-1}||\leq\mu^{-1}$, and
$||B_m||\leq\lambda$ and $\alpha_m(0)=0$, $\beta_m(0)=0$.

Then for $0<\gamma<\min(1,\sqrt{\mu/\lambda}-1)$ and
\begin{equation}\label{cota-delta}
0<\delta<\min\left(\frac{\mu-\lambda}{\gamma+2+\gamma\sp{-1}},
\frac{\mu-(1+\gamma)\sp2\lambda}{(1+\gamma)(\gamma\sp2+2\gamma+2)}
\right)
\end{equation}
we have the following property: If $||\alpha_m||_{C^1}<\delta$ and
$||\beta_m||_{C^1}<\delta$ for all $m\in \Z$ then there is
\begin{itemize}
\item[(1)] a unique family $\{W^+_m\}_{m\in\Z}$ of
$l$-dimensional $C^1$ manifolds
\[W^+_m=\{(x,\varphi^+_m(x))\, :\, x\in\C^l\}=\graph\varphi^+_m\] and
\item[(2)] a unique family $\{W^-_m\}_{m\in\Z}$ of
$k$-dimensional $C^1$ manifolds
\[W^-_m=\{(\varphi^-_m(y),y)\, :\, y\in\C^{n-l}\}=\graph\varphi^-_m,\]
\end{itemize}
where $\varphi^+_m:\C^l\rightarrow\C^{k}$,
$\varphi^-_m:\C^{k}\rightarrow\C^{l}$,
$\sup_{m\in\Z}||D\varphi^\pm_m||<\gamma$, and the following properties
holds:
\begin{itemize}
\item[(i)] $f_m(W^-_m)=W^-_{m+1}$, $f_m(W^+_m)=W^+_{m+1}$.
\item[(ii)] The inequalities
\[||f_m(z)||<\lambda'||z|| \textrm{ for } z\in
W^-_m,\]and\[||f^{-1}_{m-1}(z)||<\mu'||z|| \textrm{ for } z\in
W^+_m\] hold, where
$\lambda'=(1+\gamma)(\lambda+\delta(1+\gamma))<\frac{\mu}{1+\gamma}
-\delta=\mu'$.
\item[(iii)] Let $\lambda'<\nu<\mu'$. If
$||f_{m+j-1}\circ\cdots\circ f_m(z)||<C\nu^j$ for all $j\geq0$ and
some $C>0$ then $z\in W^-_m$.\newline
Similarly, if $||f_{m-j}^{-1}\circ\cdots\circ
f_{m-1}^{-1}(z)||<C\nu^{-j}$ for all $j\geq0$ and some $C>0$ then
$z\in W^+_m$.
\end{itemize}
Finally, in the hyperbolic case $\lambda<1<\mu$ the families
$\{W^+_m\}_{m\in\Z}$ and $\{W^-_m\}_{m\in\Z}$ consist of $C^r$
manifolds.
\end{teo}

\begin{rem}
It is important to note that the axes $\C\sp l$ and $\C\sp k$ are not
invariant by the action of $(Df_m)_0$. However, there exist a
splitting $\C\sp n=E\sp+_m\oplus E\sp-_m$ with $\dim_\C E_m\sp+=l$ and
$\dim_\C E_m\sp-=k$, that are invariant by the action of $(Df_m)_0$,
satisfy that
\[||(Df_m)_0\sp{-1}|_{E_m\sp+}||\leq(\mu\sp\prime)\sp{-1},\, \,
\textrm{ and }\, \, ||(Df_m)_0|_{E_m\sp{-}}||\leq\lambda\sp\prime,\]
and in this case $T_0W\sp\pm_m=E\sp\pm_m$. See \cite{KH} for details.
\end{rem}

It is important also to note the following proposition, proved in
\cite{KH}.

\begin{pr}
The invariant manifolds (with $C\sp1$ topology) obtained in the
Hadamard-Perron Theorem, have continuous dependence with respect to
the family $f=\{f_m\}_{m\in\Z}$, with the $C\sp 1$ topology defined by
\[d_1(f,g)=\sup_{m\in\Z}d_{C\sp1}(f_m,g_m).\]
\end{pr}

\subsection{Sketch of proof of Hadamard-Perron
Theorem}\label{sketch.proof}

In the proof of Theorem \ref{h_p_thm} (see \cite{KH}), the functions
$\varphi^+_m$ are obtained as fixed point of a contractive operator in
a space of Lipschitz maps. We enumerate the main fact:
\begin{enumerate}
\item Let $C_\gamma^0$ the space of sequences as form
$\varphi_*=\{\varphi_m\}_{m\in\Z}$ where each $\varphi_m$ is in the
set
\[C_\gamma^0(\C^l)=\{\varphi:\C^l\rightarrow\C^{n-l}\,
:\,\textit{Lip}(\varphi)<\gamma,\textrm{ and }\varphi(0)=0\}.\]
\item The set $C_\gamma^0$ is a compact metric space with the
metric defined by
\[d_*(\varphi_*,\phi_*)=\sup_{m\in\Z}d(\varphi_m,\phi_m);\]
where
\[d(\varphi,\phi)=\sup_{x\in\C^l\setminus\{0\}}
\frac{||\varphi(x)-\phi(x)||}{||x||}\]
is a metric in $C_\gamma^0(\C^l)$. Note that $(C_\gamma^0(\C^l),d)$ is
also compact metric space.
\item The action of $f=\{f_m\}_{m\in\Z}$ in the space
$C^0_\gamma$ is the desired contraction; this action is defined as
follows: denote by $(f_m)_*\varphi$ the unique Lipschitz map that
satisfy the equation
\[f_m(\graph\varphi)=\graph((f_m)_*\varphi).\]
On the other hand, we have the bijection
$G^m_\varphi:\C^l\rightarrow\C^l$ defined by
\[G^m_\varphi(x)=A_mx+\alpha_m(x,\varphi(x)),\]and the map
$F^m_\varphi:\C^{l}\rightarrow\C^{l}$ given by
\[F^m_\varphi(x)=B_m\varphi(x)+\beta_m(x,\varphi(x)),\]
it follows that the function $(f_m)_*\varphi$ is given by the
expression
\[(f_m)_*\varphi(x)=F^m_\varphi\circ(G^m_\varphi)^{-1}(x).\]
Finally if we define $f\varphi_*=\{\psi_m\}_{m\in\Z}$, whit
$\psi_{m+1}=(f_m)_*\varphi_m$, we have that
\begin{equation}\label{limit-eq}
\lim_{n\rightarrow\infty}f^n\varphi_*=\varphi^+_*,
\end{equation}
where $\varphi\sp*\in C\sp0_\gamma$ and
$\varphi^+_*=\bpl\varphi_m\sp+\bpr_{m\in\Z}$ is the sequences of
function given by the Hadamard-Perron Theorem.
\end{enumerate}

\subsection{Technical considerations}

To apply the previous Theorem and the subsequent results, is necessary
to cons\-truct the family $\{f_m\}_{m\in\Z}$ that carries the
asymptotic information of the map $f$ along the whole orbit of some
point $x\in\Lambda$. For this construction, we assume that $f$ is
partially hyperbolic on $\Lambda$ (see Definition \ref{part_hip}) with
$||Df\sp n(x)|_{E(x)}||\leq C\widetilde{\lambda}\sp n$ and
$||Df\sp{-n}(x)|_F(x)||\leq C\widetilde{\mu}\sp{-n}$ for all $n\geq0$.

First one, note that given $\delta>0$ we can find $R>0$ such that
for every $x_0\in\Lambda$ we can write
\[f(x)=f(x_0)+Df(x_0)(x-x_0)+R_{x_0}(x-x_0)\]
on $\C\sp n$, and $||R_{x_0}(x-x_0)||_{C\sp1}<\delta$ for all
$x\in\Delta_n(x_0,2R)$.

Moreover, the following statement hold.

\begin{lm} For every $\delta>0$, there exist $R>0$ uniformly in
$\Lambda$, and smooth diffeomorphisms
$f_{x_0}:\C\sp n\rightarrow\C\sp n$ for $x_0\in\C\sp n$, such that
$f_{x_0}(0)=0$,
\[f_{x_0}(h)=f(x+h)-f(x_0)\, \, \, \textit{ for all
}h\in\Delta_n(0,R),\] and $||f_{x_0}(h)-Df(x_0)(h)||_{C\sp1}<\delta$
for all $h\in\C\sp n$. Moreover, we can construct this family so that
the functions $f_{x_0}$ depend continuously in the $C\sp 1$ topology,
of the point $x_0$.
\end{lm}
\begin{proof}
Given $R>0$ take $\rho:\C\sp n\rightarrow[0,1]$ a smooth function
such that $\rho=1$ on $\Delta_n(x_0,R)$, $\rho=0$ outside of
$\Delta_n(x_0,2R)$ and $||D\rho||\leq C/R$. So
defining\[f_{x_0}(h)=\rho(h)(f(h+x_0)-f(x_0))+(1-\rho(h))Df(x_0)\cdot
h\] we conclude this proof, once we choose $R>0$ small.
\end{proof}

Now taking $L_{x_0}:\C\sp n=\C\sp l\oplus\C\sp k\rightarrow \C\sp n$
a linear orthogonal complex map such that $L_{x_0}(\C\sp l)=F(x_0)$
and $L_{x_0}(\C\sp k)=F(x_0)\sp\perp$, and define the maps
$\widehat{f}_{x_0}=L_{f(x_0)}\sp{-1}\circ f_{x_0}\circ L_{x_0}$, then
$\widehat{f}_{x_0}$ has the form
\[\widehat{f}_{x_0}(x,y)=(A_{x_0}x+\alpha_{x_0}(x,y),B_{x_0}
y+\beta_{x_0}(x, y)).\]
To finish, we denote $x_m=f\sp m(x_0)$ with $m\in\Z$ and
$f_m=\widehat{f}_{x_m}$, then:
\begin{itemize}
\item[{\it 1)}] $f_m$ is holomorphic in $\Delta_n(0,R')$ for
every $R'<R$,
\item[{\it 2)}] since that the splitting $T_\Lambda\C\sp
n=E\oplus F$
varies continuously, and the angle between the subspaces $F$ and
$E$ are uniformly away from zero (see \cite{pesin} for instance), it
follows that there exist $\lambda<\widetilde{\mu}$ such that are
satisfied the hypothesis of Hadamard-Perron Theorem.
\item[{\it 3)}] it follows from the previous construction that
the correspondence $x_0\mapsto\{f_m\}_{m\in\Z}$ is continuous in the
$C\sp 1$ topology.
\end{itemize}

\subsection{Proof of Theorem \ref{hol1}}

To proof the Theorem \ref{hol1}, is only necessary to observe the
following Proposition.

\begin{pr}\label{pr.holm}
Under the hypothesis of Theorem \ref{h_p_thm}, suppose that the
fo\-llo\-wing additional conditions hold:
\begin{enumerate}
\item $\mu>1$.
\item There exists $R>0$ such that, for each $m\in\Z$, the map
$f_m$ is holomorphic in some neighborhood of the closed polydisc
$\Delta_n(0,R)\subset\C^n$.
\end{enumerate}
Then there exists $0<r< R$ such that each $\varphi^+_m$ is holomorphic
in some neighborhood of $\Delta_l(0,r)\subset\C^l$, where
$\varphi^+_m$ is as in (1) in the Theorem \ref{h_p_thm}.
\end{pr}

So in the hypothesis of the Theorem \ref{hol1}, is only necessary to
work with $f\sp{-1}$ instead $f$, and construct the family as in the
previous subsection.

\begin{rem}
In several works (see for example \cite{bs1} or \cite{bs2}), is
proved the holomorphy of the stable/unstable manifolds under the
hypothesis of hyperbolicity in the compact invariant set. In our case,
we only consider partially hyperbolic map with unstable direction.
\end{rem}

\begin{proof}[{\bf Proof of Proposition \ref{pr.holm}}]
Denote by $\mathcal{O}^0_\gamma(r)\subset C^0_\gamma$, the set of
sequences of functions that are holomorphic in some neighborhood of
the closed polydisc $\Delta_l(0,r)$ in each level $m\in\Z$. To prove
the Proposition, is only necessary to prove that there exists
$0<r< R$ such that:
\begin{itemize}
\item[(a)] $\mathcal{O}^0_\gamma(r)$ is a closed space in
$C^0_\gamma$,
\item[(b)] $\mathcal{O}^0_\gamma(r)$ is invariant by the action
$f$.
\end{itemize}
If we assume that (a) and (b) holds, and since that equation
(\ref{limit-eq}) hold for every $\varphi_*\in\mathcal{O}^0_\gamma(r)$,
the limit
\[\lim_{n\rightarrow\infty}f^n\varphi_*=\varphi^+_*\]
there exists and is an element of $\mathcal{O}^0_\gamma(r)$, so each
function $\varphi^+_m$ is holomorphic in some neighborhood of
$\Delta_l(0,r)$.

Observe that for proof the two previous assertions, is only necessary
proof that:
\begin{itemize}
\item[(a')] $\mathcal{O}^0_\gamma(r,\C^l)$ is a closed space in
$C^0_\gamma(\C^l)$,
\item[(b')] $\mathcal{O}^0_\gamma(r,\C^l)$ is invariant by the
action $f_m$, for all $m\in\Z$.
\end{itemize}
where $\mathcal{O}^0_\gamma(r,\C^l)$ is the subset of
$C^0_\gamma(\C^l)$, whose elements are holomorphics function in some
neighborhood of the polydisc $\Delta_l(0,r)$.

The first assertion (a'), follows after observing that the metric
defined in the paragraph (2.) of the section \ref{sketch.proof},
induce the uniformly convergence on compact topology in
$\mathcal{O}^0_\gamma(r,\C^l)$, so if
$\varphi_n\in\mathcal{O}^0_\gamma(r,\C^l)$ and
$\varphi_n\rightarrow\varphi$ for some $\varphi\in C^0_\gamma(\C^l)$
then, the limit map $\varphi$ is an element of the set
$\mathcal{O}^0_\gamma(r,\C^l)$.

The proposition (b'), it follows from the following: in the proof of
the Theorem \ref{h_p_thm} in \cite{KH}, we can see that
\begin{equation}\label{G_ineq}
||G^m_\varphi(x)||\geq\mu_0||x||.
\end{equation}
where the constant is $\mu_0=(\mu-\delta(1+\gamma))$. This constant is
greater than 1 if and only if, $\mu>1$ and $\delta$ and $\gamma$ are
small enough. If we take $r=\mu_0^{-1}R$, the functions $F^m_\varphi$
and $G^m_\varphi$ are holomorphics in some neighborhood of
$\Delta_l(0,r)$ when $\varphi\in \mathcal{O}^0_\gamma(r,\C^l)$. It
follows by the equation (\ref{G_ineq}) that $\Delta_l(0,R)\subset
G^m_\varphi(\Delta_l(0,r))$, then the function $(G^m_\varphi)^{-1}$ is
holomorphic in $\Delta_l(0,r)$, and also by equation (\ref{G_ineq}),
it follows that
\[(G^m_\varphi)^{-1}(\Delta_l(0,r))\subset\Delta_l(0,\mu^{-1}
_0r)\subset\Delta_l(0,r).\]
We obtain that $F^m_\varphi\circ(G^m_\varphi)^{-1}$ is holomorphic in
some neighborhood of $\Delta_l(0,r)$, is as desired.
\end{proof}

\section{Dynamically defined and Overlapping
pro\-per\-ty}\label{section4}

Now we return to the original context exposed in the Section
\ref{preliminaries}. The map $f:\C\sp n\rightarrow\C\sp n$ is a
biholomorphism partially hyperbolic on a compact $f$-invariant set
$\Lambda$ with splitting $T_\Lambda\C\sp n=E\sp s\oplus F$. We recall
the existence of the $cu$-leaf $\www{cu}{x}{\e}$ for every
$x\in\Lambda$ and $0<\e\leq1$, that are locally $f$-invariant.

On the other hand, the notion of $cu$-leaves dynamically defined, say
that the $cu$-leaf are locally, the local unstable set. Then is
natural to expect that if the $cu$-leaf are dynamically defined, they
have a similar asymptotic behavior than the unstable set. This is
exemplified in the following Lemma.

\begin{lm}\label{equiv_cond_dyn_def}
The $cu$-leaves are dynamically defined, if and only if, there exists
$r\ll1$ such that for all $x\in\Lambda$, the following statement
holds:
\begin{enumerate}
\item For any $r_1<r$, there exist $r_0<r_1$ such that for
every $n\geq0$ and $x\in\Lambda$, $f\sp{-n}(\www{cu}{x}{r_0})\subset
\www{cu}{f\sp{-n}(x)}{r_1}$.
\item For every $r_1<r$ and $r_0<r_1$, there exists
$N=N(r_0,r_1)$ such that for all $x\in \Lambda$ and $n\geq N$
$f\sp{-n}(\www{cu}{x}{r_1})\subset\www{cu}{f\sp{-n}(x)}{r_0}$.
\end{enumerate}
\end{lm}
\begin{proof}
This is elementary, and the proof is left to the reader.
\end{proof}

The first statement in the previous Lemma say that the local
$cu$-leaf not grow to the past, then is always contained in a leaf
of fixed size, and the second say even more: the local $cu$-leaf
become small after a fixed number of iterates to the past.

We can do a more detailed description of the asymptotic behavior of
the $cu$-leaf. For this this we introduce the notion of {\it
overlapping property}.

\begin{defi}
Given a number $r>0$, we say that the $cu$-leaves has the overlapping
property for $r$, if for every $0<r_2<r$ there exist
$0<r_{-1}<r_0<r_1<r_2$, a number $N=N(r_0,r_1)$ and closed topological
balls $B\sp{cu}(x)$ with $\www{cu}{x}{r_{-1}}\subset
B\sp{cu}\subset(\www{cu}{x}{r_0})\sp\circ$ for every $x\in \Lambda$,
such that the following statement holds:
\begin{enumerate}
\item For every $n\geq N$ we have the inclusion
$f\sp{-n}(\www{cu}{x}{r_1})\subset(\www{cu}{f\sp{-n}(x)}{r_{0}}
)\sp\circ$,
\item $\www{cu}{f\sp{-N}(x)}{r_{-1}}\subset
f\sp{-N}(\www{cu}{x}{r_1})\subset
(B\sp{cu}(f\sp{-N}(x)))\sp\circ$,
\item For every $0\leq k\leq N$, we have
$f\sp{k}(B\sp{cu}(f\sp{-N}(x)))\subset(\www{cu}{f\sp{N-k}(x)}{r}
)\sp\circ$.
\end{enumerate}
\end{defi}

The overlapping is produced by the topological balls $B\sp{cu}(x)$
after a finite number of iterations to the future. Moreover,
the previous definition establish that the size of the topological
balls increase because
\[B\sp{cu}(f\sp N(x))\subset\www{cu}{f\sp N(x)}{r_1}\subset f\sp
N(B\sp{cu}(x))\sp\circ,\]
but do not excessively (property 3 in the previous definition). Also
note that we require that the size of the balls be in some sense
uniform on $x$ (the condition $\www{cu}{x}{r_{-1}}\subset
B\sp{cu}\subset(\www{cu}{x}{r_0})\sp\circ$).

The main result of this section is to proof the following Proposition.

\begin{pr}\label{pr.cueq}
If the $cu$-leaves are dynamically defined, then there exists $r>0$
such that such that the $cu$-leaves has the overlapping property for
$r$.
\end{pr}
\begin{proof}
Let $r>0$ has in the Lemma \ref{equiv_cond_dyn_def}\,. If let us take
$r_2<r$, then for every $x\in \Lambda$,
\[\d\left(\partial\www{cu}{x}{r_2},\partial\www{cu}{x}{r}\right)>0,\]
where $\d$ is the induced distance in the center-unstable leaf.
It follows by compactness of $\Lambda$ and continuity of the
$cu$-leaves, that there exist a positive number $\delta>0$ such that
\[\d\left(\partial\www{cu}{x}{r_2},\partial\www{cu}{x}{r}
\right)>\delta.\]

Now let us take $r_1<r_2$ as in the item 1 in the previous Lemma.
Since that for every $n\geq0$,
$f\sp{-n}(\www{cu}{x}{r_1})\subset\www{cu}{f\sp{-n}(x)}{r_2}$ we
have in particular that
\[
\begin{aligned}
\d\left(\partial f\sp{-n}(\www{cu}{x}{r_1}),\partial
\www{cu}{f\sp{-n}(x)}{r}\right)&
\geq\d\left(\partial\www{cu}{f\sp{-n}(x)}{r_2},\partial\www{cu}{f\sp{
-n}(x)}{r}\right)\\&>\delta.
\end{aligned}\]

If we take $r_0<r_1$, and $\e$ small enough such that $r_0-\e>0$, we
know from the item 2 in the previous Lemma, that there exit
$N=N(r_0-\e,r_1)$ such that
\[f\sp{-n}(\www{cu}{x}{r_1})\subset \www{cu}{f\sp{-n}(x)}{r_{0}-\e}
\subset(\www{cu}{f\sp{-n}(x)}{r_{0}})\sp\circ\]
for every $n\geq N$, and this implies the first item.

On the other hand, we can define the function
\[\rho(x)=\d\left(f\sp{-N}(x),\partial
f\sp{-N}(\www{cu}{x}{r_1})\right)>0\] that is
continuous in $\Lambda$. Let $\rho_0=\inf_{x\in\Lambda}\rho(x)$. Then
for every $x$ there exist a neighborhood $U_x$ and a radius $r_x$
such that for $y\in U_x$
\[\d\left(\www{cu}{f\sp{-N}(y)}{r_x},\partial
f\sp{-N}(\www{cu}{y}{r_1})\right)>\frac{\rho_0}{2}.\]
So by compactness, there exist a $r_{-1}$ such that
\[\d\left(\www{cu}{f\sp{-N}(x)}{r_{-1}},\partial
f\sp{-N}(\www{cu}{x}{r_1})\right)>\frac{\rho_0}{2}\]
in $\Lambda$ and in particular, $\www{cu}{f\sp{-N}(x)}{r_{-1}}\subset
f\sp{-N}(\www{cu}{x}{r_1})$, that is the first inclusion of the second
item.

For the second inclusion of the item 2 and the item 3, we must first
construct the sets  $B\sp{cu}(x)$. For this, let us take
\[B(x)=\bpl z\in \www{cu}{x}{r}\, :\,
\d\left(z,\partial\www{cu}{x}{r}\right)
\geq\delta/2\bpr.\]
Then it is clear that $f\sp{-n}(\www{cu}{x}{r_1})\subset
(B(f\sp{-n}(x)))\sp\circ$ for all $n\geq0$, thus we define
$B\sp{cu}(f\sp{-N}(x))$ has the connected component that contain
$f\sp{-N}(\www{cu}{x}{r_1})$ of the intersection
\[\www{cu}{f\sp{-N}(x)}{r_{0}-\e})\cap
f\sp{-1}(B(f\sp{-(N-1)}(x)))\cap\ldots\cap
f\sp{-N}(B(x)).\]
By construction the set $(B\sp{cu}(f\sp{-N}(x)))\sp\circ$ contain
$f\sp{-N}(\www{cu}{x}{r_1})$ and it follows the third item, that
conclude the proof of this Lemma.
\end{proof}

\begin{rem}\label{r0_peq}
It is important to recall that the election of the constant
$r_0<r_1$ is arbitrary, once we take $r_1<r_2$.
\end{rem}

\section{Holomorphic center-unstable submanifolds}

This section is devote to prove the following Theorem.

\begin{teo}\label{teo.cueq}
If the $cu$-leaf are dynamically defined, then they are holomorphic
submanifolds of $\C\sp n$.
\end{teo}

For this, we use the notation and the technique introduced in the
Section \ref{section3}. The principal key is the overlapping property
in the $cu$-leaf. We rewrite the definition of overlapping property in
Hadamard-Perron notation.

Let $f=\{f_m\}_{m\in\Z}$ and $\varphi\sp+=\{\varphi\sp+_m\}_{m\in\Z}$
the families of function as in the Theorem \ref{h_p_thm}\,. We recall
that the graph of the functions $\varphi\sp+_m$ are, by some local
change of chart, the center-unstable manifolds given in the Theorem
\ref{hol0}.

\begin{defi}We say that the family $\varphi\sp+$ as the overlapping
property for $r>0$ if: for every $r_2<r$ there
exist $r_{-1}<r_0<r_1<r_2$, an integer $N=N(r_0,r_1)$, and a family of
closed topological balls $U_m$ with $\Delta_l(0,r_{-1})\subset
\overline{U_m}\subset \Delta_l(0,r_0)$, such that if we denote
$D\sp+_m=\graph(\varphi_m\sp+|{U_m})$ and
$W\sp+_m(r')=\graph(\varphi\sp+_m|{\Delta_l(0,r')})$ are satisfied
the following properties:
\begin{itemize}
\item[{\it a)}] For every $n\geq N$ and $m\in\Z$, hold that
$f\sp{-1}_{m-n}\circ\cdots\circ f\sp{-1}_{m-1}(W\sp+_m(r_1))\subset
(W\sp+_m(r_0))\sp\circ$,
\item[{\it b)}] $W\sp+_{m-N}(r_{-1})\subset
f\sp{-1}_{m-N}\circ\cdots\circ
f\sp{-1}_{m-1}(W\sp+_m(r_1)) \subset (D_{m-N}\sp+)\sp\circ$ for every
$m\in\Z$,
\item[{\it c)}] For every $0\leq k\leq N-1$ we have
\[f_{m-(N-(k-1))}\circ\cdots\circ f_{m-(N-1)}\circ
f_{m-N}(D_{m-N}\sp+)\subset W\sp+_{m-(N-k)}(r)\sp\circ\] for every
$m\in\Z$,.
\end{itemize}
\end{defi}

From the Proposition \ref{pr.cueq}, $cu$-dynamically defined implies
the overlapping property, hence the proof the Theorem \ref{teo.cueq}
it follows directly from the following Proposition.

\begin{pr}\label{ws_holm}
Under the hypothesis of Theorem \ref{h_p_thm}, suppose that the
fo\-llo\-wing additional conditions hold:
\begin{enumerate}
\item There exists $R>0$ such that, for each $m\in\Z$, the map
$f_m$ is holomorphic in some neighborhood of the closed polydisc
$\Delta(0,R)\subset\C^n$,
\item For every $0<r<R$, the family $\varphi\sp+$ as the
overlapping property for $r>0$
\end{enumerate}
then there exit $R'<R$ such that $\varphi\sp+_m$ is holomorphic in
$\Delta_l(0,R')$.
\end{pr}

Now, we want to highlight the main difference of the Proposition
\ref{pr.holm} with the Proposition \ref{ws_holm}. In the first of
them, it is assumed that $F$ is an unstable direction. Here we only
assume that the center-unstable manifold are dynamically defined.

However, the states of the proof has many similarities. The goal is to
show that using the graph transform operator it is possible to prove
that the $cu$-leaves are limits of the graph of uniformly bounded
holomorphic function, and therefore it is also holomorphic. The main
difficulty is to show that only using the dynamically defined
property, is arrange to recover, after some iterate, the overlapping
property of the graph transform operator.

\begin{proof}[{\bf Proof of Proposition \ref{ws_holm}}]We use the same
notation of the proof of Proposition
\ref{pr.holm}. Firstly let us take $r_2<r<R$ with
\begin{equation}\label{r2}
2\gamma r_2<\frac{R-r}{2}.
\end{equation}
We recall that from the Remark \ref{r0_peq}, we can take $r_0<r_1$
small enough such that
\begin{equation}\label{r0}
2\gamma r_0<\frac{r_1-r_0}{2}.
\end{equation}
The proof goes through a series of claims.

\noindent
{\bf Claim 1:}\, {\it There exists $\lambda_0<1$, such that for every
$\varphi,\, \phi\in C\sp0_\gamma(\C\sp l)$, $m\in\Z$ and $x\in\C\sp l$
we have the inequality}
\[||f_m(x,\varphi(x))-f_m(x,\phi(x))||\leq\lambda_0||
\varphi(x)-\phi(x)||.\]
\begin{proof}[{\bf Proof of Claim 1}] We recall that
$f_m(x,\varphi(x))=(G\sp m_\varphi(x),F\sp m_\varphi(x))$, then we
have
\[
\begin{aligned}
||f_m(x,\varphi(x))&-f_m(x,\phi(x))||\\
&\, \, \, \, \, \, \, \, \, \, \, \, \, \, \, \, \, \, \, =||(G\sp
m_\varphi(x),F\sp
m_\varphi(x))-(G\sp m_\phi(x),F\sp m_\phi(x))||\\
&\, \, \, \, \, \, \, \, \, \, \, \, \, \, \, \, \, \, \, \leq||G\sp
m_\varphi(x)-G\sp m_\phi(x)||+||F\sp m_\varphi(x)-F\sp
m_\phi(x)||\\
&\, \, \, \, \, \, \, \, \, \, \, \, \, \, \, \, \, \, \,
\leq||(A_mx+\alpha_m(x,\varphi(x)))-(A_mx+\alpha_m(x,\phi(x)))||\\
&\, \, \, \, \, \, \, \, \, \, \, \, \, \, \, \, \, \, \, \, \, \,
+||(B_m\varphi(x)+\beta_m(x,\varphi(x)))-(B_m\phi(x)+\beta_m(x,
\phi(x)))||\\
&\, \, \, \, \, \, \, \, \, \, \, \, \, \, \, \, \, \, \,
\leq||\alpha_m(x,\varphi(x))-\alpha_m(x,
\phi(x))||+||B_m(\varphi(x)-\phi(x))||\\
&\, \, \, \, \, \, \, \, \, \, \, \, \, \, \, \, \, \, \, \, \, \,
+||\beta_m(x,\varphi(x))-\beta_m(x,\phi(x))||\\
&\, \, \, \, \, \, \, \, \, \, \, \, \, \, \, \, \, \, \,
\leq(\lambda+2\delta)||\varphi(x)-\phi(x)||,
\end{aligned}
\]
then let us take $\lambda_0=(\lambda+2\delta)$, and we will prove that
$\lambda_0<1$. Firstly note that we can assume that $\mu\leq1$, if not
by Proposition \ref{pr.holm} it follows that $\varphi\sp+_m$ is
holomorphic in a polydisc $\Delta_l(0,R')$ for some $R'<R$, that is
we want to prove. On other hand, by inequality (\ref{cota-delta}) in
the Theorem \ref{h_p_thm}
\[\delta<\frac{\mu-\lambda}{\gamma+\gamma\sp{-1}+2},\]
and this is less than $(1-\lambda)/2<1$. This end the
proof of the claim.
\end{proof}

In that follows, we fix $m\in \Z$ and define
\begin{equation}\label{iterate}g_k=f_{(m+kN)+(N-1)}\circ
f_{(m+kN)+(N-2)}\circ\cdots\circ
f_{(m+kN)+1}\circ f_{(m+kN)}.
\end{equation}
Then we can write $g_k$ as the form
\[g_k(x,y)=(C_kx+c_k(x,y),D_ky+d_k(x,y)),\]
where
\[C_k=A_{(m+kN)+(N-1)}\cdot A_{(m+kN)+(N-2)}\cdot\ldots\cdot
A_{(m+kN)+1}\cdot A_{(m+kN)}\]
and
\[D_k=B_{(m+kN)+(N-1)}\cdot B_{(m+kN)+(N-2)}\cdot\ldots\cdot
B_{(m+kN)+1}\cdot B_{(m+kN)}.\]

We recall that graph transform operator $(f_m)_*$ of a Lipschitz
function $\varphi$
is defined by the equation
\[(x',(f_m)_*\varphi(x'))=f_m(x,\varphi(x))=(A_mx+\alpha_m(x,
\varphi(x)),B_m\varphi(x)+\beta_m(x,\varphi(x))).\]
It is possible to prove that the map $G^m_\varphi:\C^l\rightarrow\C^l$
given by
\[G^m_\varphi(x)=A_mx+\alpha_m(x,\varphi(x)),\] is a bijection, and
that if we define
$F^m_\varphi:\C^{l}\rightarrow\C^{l}$ by
\[F^m_\varphi(x)=B_m\varphi(x)+\beta_m(x,\varphi(x)),\]
then the graph transform operator $(f_m)_*\varphi$, is given by the
expression
\[(f_m)_*\varphi(x)=F^m_\varphi\circ(G^m_\varphi)^{-1}(x).\]
Similarly, we denote by
\[\widetilde{G}^k_\varphi(x)=C_kx+c_k(x,\varphi(x)),\] and
\[\widetilde{F}^k_\varphi(x)=D_k\varphi(x)+d_k(x,\varphi(x)),\]
the coordinates maps related with $g_k$ and $\varphi$.

For a fixed $k$ and $\varphi$, we denote:
\begin{enumerate}
\item $\varphi_1=(f_{m+kN})_*\varphi$,
\item $\varphi_{j+1}=(f_{m+kN+j})_*\varphi_{j}$, for every
$j=1,\ldots,N-2$,
\item $G_j=G\sp{m+kN+j}_{\varphi_{j}}$, for every $j=1,\ldots,N-1$.
\end{enumerate}

\noindent
{\bf Claim 2:}\, {\it We have that}
\[\widetilde{G}^k_\varphi=G_{N-1}\circ G_{N-2}\circ\ldots G_{1}\circ
G\sp{m+kN}_{\varphi},\]
{\it and the graph transform operator of $g_k$, given by equality
\ref{iterate}, is equal to}
\[(g_k)_*=(f_{(m+kN)+(N-1)})_*(f_{(m+kN)+(N-2)})_*\ldots
(f_{(m+kN)+1})_*(f_{(m+kN)})_*.\]
\begin{proof}
This is elementary, and the proof is left to the reader.
\end{proof}
As  a consequence of the previous claim, we conclude that
$\widetilde{G}\sp k_\varphi$ is a bijection of $\C\sp l$, and that
the graph transform operator related with $g_k$ is given by the
equality
\[(g_k)_*\varphi(x)=\widetilde{F}^k_\varphi\circ(\widetilde{G}
^k_\varphi)^{-1}(x).\]
In that follows by simplicity, we will work with $m=k=0$, and
the function $g_0=f_{N-1}\circ\cdots\circ f_0$, but all the following
results are true for any $m$ and $k$.\\

\noindent
{\bf Claim 3:}\, {\it If $\varphi\in C\sp0_\gamma(\C\sp l)$ is
holomorphic in some neighborhood of $U\sp+_{0}$, then the function
$\widetilde{G}\sp0_\varphi$ is holomorphic in some neighborhood of
$U\sp+_{0}$.}
\begin{proof}
From the inequality (\ref{r2}), for any point $x$ in the closed
polydisc $\Delta_l(0,r_2)$, we have that
\[||\varphi_0\sp+(x)-\varphi(x)||\leq 2\gamma||x||\leq2\gamma
r_2<\frac{R-r}{2}.\]
We recall that each map $f_j$ is holomorphic in the closed polydisc
$\Delta(0,R)$. From the item (c), it follows that
$f_0(D\sp+_0)\subset (W\sp+_1(r))\sp\circ$ and we conclude that
$G\sp0_{\varphi\sp+_0}(U_0)\subset\Delta_l(0,r)\sp\circ$. It follows
from the Claim 1 that for every $x\in U_0$
\[
\begin{aligned}
||G\sp0_{\varphi\sp+_0}(x)-G\sp0_\varphi(x)||&\leq||f_0(x,
\varphi\sp+_0(x))-f_0(x,
\varphi(x))||\\
&\leq\lambda_0||\varphi_0\sp+(x)-\varphi(x)||\\
&<\lambda_0
\frac{R-r}{2},
\end{aligned}\]
this implies that
\[||G\sp0_\varphi(x)||\leq
||G\sp0_{\varphi\sp+_0}(x)-G\sp0_\varphi(x)||+||G\sp0_{\varphi\sp+_0}
(x)||<\frac{R-r}{2}+r=\frac{R+r}{2}<R.\]
As before, we denote $\varphi_1=(f_{0})_*\varphi$,
$\varphi_{j+1}=(f_{j})_*\varphi_{j}$, for every $j=1,\ldots,N-2$; and
$G_j=G\sp{j}_{\varphi_{j}}$, for every $j=1,\ldots,N-1$. Again by item
(c), for every $1\leq k\leq N-1$ we have
$f_k\circ\cdots\circ f_0(D\sp+_0)\subset (W\sp+_{k+1}(r))\sp\circ$,
it follows that $G\sp k_{\varphi\sp+_k}\circ\cdots\circ G\sp
0_{\varphi\sp+_0}(U_0)\subset \Delta_l(0,r)\sp\circ$. We use the
following notation:
\[x\sp+_k=G\sp k_{\varphi\sp+_k}\circ\cdots\circ
G\sp0_{\varphi\sp+_0}(x)\, \, \, \, \textrm{ and }\, \, \, \,
x_k=G_{k}\circ\ldots\circ G\sp{0}_{\varphi}(x).\]
Then as before, we conclude that for every $x\in U_0$
\[
\begin{aligned}
||x\sp+_k-x_k||&\leq||f_k(x\sp+_{k-1},\varphi\sp+_k(x\sp+_{k-1}
))-f_k(x_{k-1},\varphi_k(x_{k-1}))||\\
&\leq \lambda_0||\varphi\sp+_k(x\sp+_{k-1}
)-\varphi_k(x_{k-1})||\\
&\leq\lambda_0||f_{k-1}(x\sp+_{k-2},\varphi\sp+_{k-1}(x\sp+_{k-2}
))-f_{k-1}(x_{k-2},\varphi_{k-1}(x_{k-2}))||\\
&\, \, \, \, \, \, \, \, \, \, \, \, \, \, \, \, \, \, \, \, \, \, \,
\, \, \, \, \, \, \, \, \, \, \, \, \, \vdots\\
&\leq\lambda_0\sp k||\varphi_0\sp+(x)-\varphi(x)||\\
&<\lambda_0\sp k\frac{R-r}{2}
\end{aligned}
\]
then is follows that
\[||G_{k}\circ\ldots\circ
G\sp{0}_{\varphi}(x)||\leq||x\sp+_k-x_k||+||x\sp+_k||<\frac{R-r}{2}
+r=\frac{R+r}{2}<R.\]

To end, since that $\varphi$ is holomorphic in some neighborhood of
$U_0$, $f_0$ is holomorphic in $\Delta_l(0,R)$ and
$\overline{\im(G\sp{0}_{\varphi}(U_0))}\subset\Delta_l(0,R)\sp\circ$
it follows that the map $\varphi_1$ is holomorphic in some
neighborhood of $\im(G\sp{0}_{\varphi}(U_0))$. Similarly, since that
$\overline{G_1(\im(G\sp{0}_{\varphi}(U_0)))}\subset\Delta_l(0,
R)\sp\circ$, and $f_1$ is holomorphic in this domain, we conclude that
$\varphi_2$ is holomorphic in
$\im(G_1(\im(G\sp{0}_{\varphi}(U_0))))$, and so on. This implies that
the map $\widetilde{G}\sp0_\varphi=G_{N-1}\circ\cdots G_1\circ
G\sp0_\varphi$ is holomorphic in $U_0$ and
$\overline{\im(\widetilde{G}\sp0_\varphi(U_0))}\subset
\Delta_l(0,R)\sp\circ$.
\end{proof}

\noindent
{\bf Claim 4:}\, {\it The image of $U_0$ from the map
$\widetilde{G}\sp0_\varphi$, contain the polydisc $\Delta_l(0,r_0)$.}
\begin{proof}
From the item (b), we have that $W\sp+_N(r_1)\subset
(g_0(D_0\sp+))\sp\circ$, and we recall that $g_0(D_0\sp+)$ is a
topological ball that contain $0$. Now for a point $x\in
pr_1(g_0\sp{-1}(W\sp+_N(r_1)))\subset\Delta_l(0,r_0)$ we have that
$||G\sp{N-1}_{\varphi\sp+_{N-1}}
\circ\cdots\circ G\sp0_{\varphi\sp+_0}(x)||\leq r_1$
and that
\[
\begin{aligned}
||G\sp{N-1}_{\varphi\sp+_{N-1}}\circ\cdots\circ
G\sp0_{\varphi\sp+_0}(x)
-\widetilde{G}\sp0_\varphi(x)||&\leq||g_0(x,\varphi\sp+_0(x))-g_0(x,
\varphi(x))||\\
&<\lambda_0\sp N2\gamma r_0\\
&<\frac{r_1-r_0}{2},
\end{aligned}
\]and this
last inequality comes from the inequality (\ref{r0}). This
conclude the proof of the claim.
\end{proof}
From the previous claim, in particular we have that $U_N\subset
\im(g_0(D_0\sp+))$. Since $\widetilde{F}\sp0_\varphi$ be a
holomorphic map in some neighborhood of $U_0$, and
$(\widetilde{G}\sp0_\varphi)\sp{-1}$ is holomorphic in
$\Delta_l(0,r_0)\supset U_N$ (and this because
$\widetilde{G}\sp0_\varphi$ is holomorphic and injective), it follows
that the map
$\varphi'(x)=(g_0)_*\varphi(x)=\widetilde{F}\sp0_\varphi\circ
(\widetilde{G}\sp0_\varphi)\sp{-1}(x)$ is holomorphic in $U_N$.

We conclude that for any $m$, the action of the graph transform
operator associated
with the family $g=\bpl g_k\bpr_{k\in\Z}$ defined as in the equation
(\ref{iterate}), leaves invariant the set of sequences of Lipschitz
functions that in each level is holomorphic in some neighborhood of
the sets $U$'s; and note that this set contain the linear maps.
Passing to limit, we conclude that each $\varphi\sp+_m$ is
holomorphic in the set $U_m\supset\Delta_l(0,r_{-1})$. Thus taking
$R'=r_{-1}$, we completed the proof of the Proposition.
\end{proof}

\section{Forward Expansiveness in the center-unstable leaf}

In this section we will prove the following Theorem.

\begin{teo}\label{teo6.1}
If $f$ is $cu$-forward expansive then the $cu$-leaf are dynamically
defined.
\end{teo}

For this purpose, is only necessary to prove that to be satisfied the
equivalents condition in the Lemma \ref{equiv_cond_dyn_def}, which are
proved in the following Propositions.

\begin{pr}
Let $f$ be a forward expansive map in the $cu$-leaves, with
constant of expansiveness $c$. Then for every $r_1<c$ there exist
$r_0<r_1$ such that for all $x\in\Lambda$ and $n\geq0$\[f\sp{-n}
(\www{cu}{x}{r_0})\subset\www{cu}{f\sp{-n}(x)}{r_1}.\]
\end{pr}

\begin{proof}
We suppose that is not true, thus there exists $r_1$ such that
the previous proposition not holds. Let $\rho\gtrapprox1$ such
that $\rho r_1<c$ and let $(r_k)_k$ be a sequence of positive
numbers such that $r_k\rightarrow0$ and $r_k<r_1$. Thus there
exist $x_k\in\Lambda$ and $(n_k)_k\nearrow\infty$ such that
\[f\sp{-n_k}(\www{cu}{x_k}{r_k})\nsubseteq\www{cu}{f\sp{-n_k}
(x_k)}{r_1}\subset\www{cu}{f\sp{-n_k}(x_k)}{\rho r_1}.\]
We take each $n_k$ minimal with this property. Let us take
$y_k=f\sp{-n_k}(x_k)$ and take $z_k$ some point in the following
intersection
\[f\sp{-n_k}(\www{cu}{x_k}{r_k})\cap \overline{\www{cu}{y_k}
{\rho r_1}\setminus\www{cu}{y_k}{r_1}}.\]
Also we take $y_0$ and $z_0$ such that $z_k\rightarrow z_0$ and
$y_k\rightarrow y_0$.
By construction (and $C\sp1$ continuity of the $cu$-leaves) we
have that
$z_0\in \overline{\www{cu}{y_0}{\rho r_1}\setminus\www{cu}{y_0}
{r_1}}$.

We assert that \[\d(f\sp n(y_0),f\sp n(z_0))\leq\rho r_1\] for
each $n\geq1$, and since
$\rho r_1<c$ we have a contradiction with the expansiveness in
the $cu$-leaves. Then to conclude the proof, is only necessary
to prove the previous assertion.

By contradiction, we assume that there exist $n$ such that
\[\d(f\sp n(y_0),f\sp n(z_0))=\gamma>\rho r_1.\] By continuity of
$f\sp n$, given $\varepsilon>0$ we can
take $k\gg 1$ such that $n_k>n$ and satisfied \[\d(f\sp n(y_k),f
\sp n(y_0))<\varepsilon\, \, \, \textrm{ and }\, \, \, \d(f\sp n
(z_k),f\sp n(z_0))<\varepsilon.\]
If we take $\varepsilon$ such that
$\gamma-2\varepsilon>\rho r_1$ we conclude that
$\d(f\sp n(z_k),f\sp n(y_k))>\gamma-2\varepsilon>\rho r_1$.
To end, taking $\widetilde{z}_k\in\www{cu}{x_k}{r_1}$ such that
$f\sp{n_k}(\widetilde{z}_k)=z_k$,
the previous inequality implies that
\[\d(f\sp{n-n_k}(\widetilde{z}_k),f\sp{n-n_k}(x_k))>\rho r_1,\]
that is
\[f\sp{n-n_k}(\www{cu}{x_k}{r_k})\nsubseteq\www{cu}{f\sp{n-n_k}(x_k)}
{r_1},\]
that contradict the minimality of $n_k$. This ends the proof.
\end{proof}

\begin{pr}
Let $f$ be a forward expansive map in the $cu$-leaves, and
$r_0<r_1$ such that $r_0\in I(r_1)$. Then for every
$0<\varepsilon<r_1<c$ there exists $N=N(\varepsilon,r_0)$ such that
for all $x\in \Lambda$ and $n\geq N$
\[f\sp{-n}(\www{cu}{x}{r_0})\subset\www{cu}{f\sp{-n}(x)}{\e}.\]
\end{pr}
\begin{proof}
We suppose that is not true. Thus there exist $\e$ such that
for all $k\geq0$ there exist $x_k\in \Lambda$ and $n_k>k$ such that
\[f\sp{-n_k}(\www{cu}{x_k}{r_0})\nsubseteq\www{cu}{f\sp{-n_k}
(x_k)}{\e}\subset\www{cu}{f\sp{-n_k}(x_k)}{r_1}.\]
We take each $n_k$ minimal with this property. Let us take
$y_k=f\sp{-n_k}(x_k)$ and take $z_k$ some point in the following
intersection
\[f\sp{-n_k}(\www{cu}{x_k}{r_k})\cap\overline{\www{cu}{y_k}{r_1}
\setminus\www{cu}{y_k}{\e}}.\] Note that in particular
$\d(y_k,z_k)<c$.

Also we take $y_0$ and $z_0$ such that $z_k\rightarrow z_0$ and
$y_k\rightarrow y_0$. By construction (and $C\sp1$ continuity of the
$cu$-leaves) we have that
$z_0\in\overline{\www{cu}{y_0}{r_1}\setminus\www{cu}{y_0}{\e}}$ and
$\d(y_0,z_0)\leq c$.

We assert that \[\d(f\sp n(y_0),f\sp n(z_0))\leq c\] for each
$n\geq1$, and since $c$ the expansiveness constant, we have a
contradiction with the hypothesis of expansiveness in the $cu$-leaves.
Then to conclude the proof, is only necessary to prove the previous
assertion.

By contradiction, and arguing as in the previous proposition, if we
assume that there exist $n$ such that $\d(f\sp n(y_0),f\sp
n(z_0))>c>\e$,there exist $k\gg 1$ such that $n_k>n$ and satisfies
$\d(f\sp n(z_k),f\sp n(y_k))>\e$. Thus
\[f\sp{n-n_k}(\www{cu}{x_k}{r_k})\nsubseteq\www{cu}{f\sp{n-n_k}(x_k)}
{r_1},\]
that contradict the minimality of $n_k$.
\end{proof}

\section{Proof of Theorem A}\label{section:7}

The proof of the Main Theorem, it follows after the following Theorem.

\begin{teo}\label{final.teo}
If the $cu$-leaf are dynamically defined, then the center-unstable
direction $F$, is an unstable direction
\end{teo}

\begin{proof}[{\bf Proof of Theorem A}]
{\it (1)\, $\Rightarrow$\, (2)} In the hyperbolic case, the $cu$-leaf
is unique and equal to the unstable manifold. The forward
expansiveness in the $cu$-leaf is a well know property of the unstable
manifold (topological expansivity).\newline

\noindent
{\it (2)\, $\Rightarrow$\, (3)} It follows from the Theorem
\ref{teo6.1}\, .\newline

\noindent
{\it (3)\, $\Rightarrow$\, (1)} It follows from the Theorem
\ref{final.teo}\, .
\end{proof}

To prove the Theorem A, it is only necessary to prove the Theorem
\ref{final.teo}, and for this, we use that the $cu$-leaf
$\www{cu}{x}{1}$ are holomorphic submanifolds of $\C\sp n$ (Theorem
\ref{teo.cueq}), biholomorphic to a polydisc.

Consider the in\-fi\-ni\-te\-si\-mal Kobayashi metric on the polydisc
that is the natural generalization of the Poincar\'e metric for the
unitary disk in several variables (see \cite{royden} for instance).
The Kobayashi metric on a polydisc
$\Delta=\Delta_1(0,r_1)\times\cdots\times\Delta_1(0,r_n)$ is given by
the equation
\[K_\Delta(x,\xi)=\max_{i}\frac{r_i|\xi_i|}{r_i\sp2-|x_i|\sp2}.\]
Then if we consider $\Delta=\Delta_k(0,r)$ we have that
$K_\Delta(0,\xi)=r\sp{-1}||\xi||$.
The following proposition is an immediate consequence of the
definition of $K_\Delta$.
\begin{pr}
Let $f:\Delta\rightarrow\Delta\sp\prime$ be a holomorphic map between
two polydisc, then
\[K_{\Delta\sp\prime}(f(x),Df_x(\xi))\leq K_\Delta(x,\xi).\]
\end{pr}

\begin{proof}[{\bf Proof of Theorem \ref{final.teo}}] Since that
the $cu$-leaf are dynamically defined without loss of generality, we
can assume that
\[f\sp{-1}(\www{cu}{x}{1})\subset\www{cu}{f\sp{-1}(x)}{1/2}\]
for every $x\in\Lambda$. Let $\phi\sp{cu}$ the continuous function
given by the Theorem \ref{hol0}. From the Theorem \ref{ws_holm}, it
follows that the function
$\phi\sp{cu}(x):\Delta_l(0,1)\rightarrow\C\sp n$ is holomorphic, where
$l$ is the complex dimension of $F(x)$.

We define the holomorphic map
$f_x:\Delta_l(0,1)\rightarrow\Delta_l(0,1)$ given by
\[f_x(z)=[(\phi\sp{cu}(f\sp{-1}(x)))\sp{-1}\circ f\sp{-1}\circ
\phi\sp{cu}(x)](z).\]
Applying the previous Proposition, it is follows that
\[
\begin{aligned}
2||Df_x(0)\xi||&=2||D[(\phi\sp{cu}(f\sp{-1}(x)))\sp{-1}](f\sp{-1}
(x))\circ Df\sp{-1}(x)\circ
D[\phi\sp{cu}(x)](0)\xi||\\
&\leq||\xi||,
\end{aligned}
\]
then
\[||Df_{f\sp{-(n-1)}(x)}(0)\circ\cdots\circ Df_{f\sp{-1}(x)}(0)\circ
Df_x(0)||\leq\left(\frac{1}{2}\right)\sp n.\]
On the other hand, using the continuity of the function
$x\mapsto\phi\sp{cu}(x)$ and the compactness of the set $\Lambda$ we
conclude that there exist a constant $C>0$ such that
$C\sp{-1}\leq||D\phi\sp{cu}(x)(0)||\leq C$.

To end, since that
\[
\begin{aligned}
Df\sp{-n}(x)|_{F(x)}=D[\phi\sp{cu}(f\sp{-(n-1)}(x))]&(0)\circ
Df_{f\sp{-(n-1)}(x)}(0)\circ\cdots\\
&\cdots\circ Df_{f\sp{-1}(x)}(0)\circ Df_x(0)\circ
D[\phi\sp{cu}(x)](x),
\end{aligned}
\]
it follows that for every $\xi\in F(x)$
\[||Df\sp{-n}(x)\xi||\leq(1/2)\sp nC\sp2||\xi||,\]as desired.
\end{proof}

\section{Some remark for complex H\'enon maps}\label{section8}

This section is devote to prove the Theorem B. Also in the end of this
section we prove the Propositions \ref{1aprop} and \ref{hol6}. For
notations and definition of the Julia set $J$ and the support of the
measure of maximal entropy $J\sp*$, see \cite{bs1}.

\subsection{Zero Lyapunov exponent measure}

In this subsection we introduce some definitions to enunciate the
Theorem B. In what follows, we assume that $f$ is a dissipative
generalized H\'enon map in $\Ctwo$, with $|\det(Df)\, |=b<1$.

Denote by $\nu$, to a $f$-invariant measure whose support
is contained in $J$. Also, we denote by $\reg(\nu)$, the set of all
regular point in $\supp(\nu)$. By the classical Oseledets Theorem, we
know that $\nu(\reg(\nu))=1$. Let $x\in J$ be a regular point and let
$\lambda\sp-(x)\leq\lambda\sp+(x)$ its Lyapunov exponents, then they
are related with a splitting $E\sp-_x$ and $E\sp+_x$ respectively.
Since $J$ has no attracting periodic points, from the equation
$\lambda\sp-(x)+\lambda\sp+(x)=\log(b)$ it follows that
$\lambda\sp-(x)\leq\log(b)<0\leq\lambda\sp+(x)$.

\begin{defi}We say that a $f$-invariant measure $\nu$:
\begin{enumerate}
\item is hyperbolic, if $\lambda\sp+(x)>0$ for $\nu$-a.e.,
\item has a zero exponent, if $\lambda\sp+(x)=0$ for $\nu$-a.e.,
%\item partially hyperbolic, if is not of the above types.
\end{enumerate}
\end{defi}

Give $\nu$ a measure, we denote by $\reg\sp+(\nu)$ (resp.
$\reg\sp0(\nu)$), the set of all regular points, that has the maximal
exponent positive (resp. null). It is clear that
$\reg(\nu)=\reg\sp+(\nu)\sqcup\reg\sp0(\nu)$, where $\sqcup$ is a
disjoint union. It is easy to see from the definition that $\nu$ is
hyperbolic (resp. be a zero exponent) if and only if
$\nu(\reg\sp+(\nu))=1$ (resp. $\nu(\reg\sp0(\nu))=1$). A measure, is
not of the above types if and only if $\nu(\reg\sp+(\nu)),\,
\nu(\reg\sp0(\nu))>0$. We recall that
$\supp(\nu)=\overline{\reg(\nu)}(\mod 0)={\reg}(\nu)(\mod 0)$.

%\begin{lm}
%For a measure that is not both hyperbolic and has a zero exponent,
%satisfies that
%$\nu(\overline{\reg\sp+}\cap\overline{\reg\sp0})=0$.
%\end{lm}
%\begin{proof}
%If not, for every Borel set $B$, we can define
%$\mu(B)=\nu(B\cap\overline{\reg\sp+}\cap\overline{\reg\sp0}
%)/\nu(\overline{\reg\sp+}
%\cap\overline{\reg\sp0})$ that define a $f$-invariant measure, then
%there exist a regular
%point in the intersection, that is a contradiction.
%\end{proof}

We can write every measure $\nu$, as a direct sum of the form
$\nu=\nu\sp+\oplus\nu\sp0$, where $\nu\sp+=\nu|_{{\reg\sp+(\nu)}}$
is hyperbolic and $\nu\sp0=\nu|_{{\reg\sp0(\nu)}}$ is has a zero
exponent. Naturally $\nu\sp0\equiv0$ when $\nu$ is hyperbolic, and
$\nu\sp+\equiv0$ when $\nu$ has a zero exponent.

\begin{rem}
It is important to recall that, for a measure that is neither
hyperbolic nor has zero exponent, the supports
$\supp(\nu\sp0)=\reg\sp0(\nu)(\mod 0)$ and
$\supp(\nu\sp+)=\reg\sp+(\nu)(\mod 0)$ can intersect, but
this intersection has measure zero both for $\nu\sp0$ and for
$\nu\sp+$.
\end{rem}

We define the set support of $J$, as the set
\[\supp(J)=\overline{\cup\bpl \supp(\nu)\,:\,\nu\, \textrm{ is }\,
f\textrm{-invariant}\, \bpr}.\]
% This set can by defined in more general context, namely, linear
% cocycles (see Chapter \ref{chap03} for details), and play an
% important role in the proof of Theorem D in this chapter.

In the paper \cite{bls1}, the authors proof that the set
$J\sp*=\supp(\mu)$, where $\mu$ is the unique measure of maximal
entropy $\log(\deg(f))$, and that any hyperbolic measure has support
contained in $J\sp*$. Then we have that
\[J\sp*=\overline{\cup\bpl \supp(\nu)\,:\,\nu\, \textrm{ is }\,
\textrm{hyperbolic}\, \bpr}.\]
Also we define the set
\[J_0=\overline{\cup\bpl \supp(\nu)\,:\,\nu\, \textrm{ is }\,
\textrm{has a zero exponent}\, \bpr}.\]
Note that by definition, $J_0$ is a compact $f$-invariant set.

\begin{pr}
The equality $\supp(J)=J\sp*\cup J_0$ holds.
\end{pr}
\begin{proof}
It is clear that $J\sp*\cup J_0\subset\supp(J)$. On the other hand,
Let $x_n\rightarrow x\in\supp(J)$ with $x_n\in\supp(\nu_n)$. Writing
$\nu_n=\nu_n\sp+\oplus\nu_n\sp0$, we have that there is an infinity
times $n$ such that either $x_n\in\supp(\nu_n\sp+)$ or
$x_n\in\supp(\nu_n\sp0)$, and we can take a subsequence converging to
$x$. This conclude the proof.
\end{proof}

\subsection{Proof of Theorem B}
This subsection in devote to prove the Theorem B, and this proof will
be supported essentially in the Forn\ae{}ss Theorem (see \cite{F}),
and the Theorem \ref{JJJ}.

\begin{teo}[{\bf Forn\ae{}ss}]
Let $f$ be a complex H\'enon map which is hyperbolic in $J\sp*$. If
$f$ is not volume preserving, then $J\sp*=J$.
\end{teo}

This implies that is sufficient to see hyperbolicity of the $J\sp*$.
This allows enunciate the following result.

\begin{teo}\label{JJJ}
Let $f$ be a complex H\'enon map, dissipative with dominated splitting
in $J\sp*$. Then we have the following dichotomy:
\begin{itemize}
\item[{\it i.}] The set $J\sp*$ is hyperbolic.
\item[{\it ii.}] $J\sp*\cap J_0\neq\emptyset$.
\end{itemize}
\end{teo}

In the next subsection, we shall prove this result, as a corollary of
the Theorem 2.1 of the celebrated work of R. Ma\~n\'e ``A proof of the
$C\sp1$ Stability Conjecture''. This Theorem can be also proved
independently of the Ma\~n\'e work. For this another proof see
\cite{pancho}.

As a corollary of the previous Theorem, we have.
\begin{cor}\label{cor_jc}
The set $J_0\neq\emptyset$ if and only if $J_0\cap
J\sp*\neq\emptyset$.
\end{cor}
\begin{proof}
If $J_0\cap J\sp*=\emptyset$, then $J\sp*$ is hyperbolic. Thus from
Forn\ae{}ss Theorem, $J$ is hyperbolic and $J_0=\emptyset$.
\end{proof}

Let $\per$ the set of all periodic point contained in $J$. From
\cite{bs1} any periodic saddle point $p$ of $f$ is on $\per$, and
$J\sp*=\overline{\per}$. We recall that from Proposition \ref{1aprop},
the dominated direction $E$ in each periodic point is a stable
direction. This justify the following definition.

\begin{defi}\label{defi:7}
\begin{enumerate}
\item We say that $\per$ is uniformly hyperbolic if there exist
a $C\geq1$ and $0<\lambda_1<1$ such that for every $n\geq1$
\[||Df\sp{-n}|_{F(p)}||\leq C\lambda_1\sp n,\]
for every $p\in\per$.
\item We say that $\per$ is uniformly expanding at the period,
if there exist a $C\geq1$ and $0<\lambda_1<1$ such that
\[||Df\sp{-\pi(p)}|_{F(p)}||\leq C\lambda_1\sp{\pi(p)},\]
for every $p\in\per$, where $\pi(p)$ is the period of $p$.
\end{enumerate}
\end{defi}

\begin{proof}[{\bf Proof of Theorem B}]{\it (1)\, $\Leftrightarrow$\,
(2)}\, From the Theorem \ref{JJJ}, it follows that if $J$ is
hyperbolic, then $J_0\cap J\sp*=\emptyset$. Thus, from Corollary
\ref{cor_jc} it follows that $J_0=\emptyset$.

The reciprocal direction, is essentially the same: Corollary
\ref{cor_jc} say that if $J_0=\emptyset$, then $J_0\cap
J\sp*=\emptyset$. The hyperbolicity of $J$, it follows from
the Theorem \ref{JJJ}, and the Forn\ae{}ss Theorem.

It is clear that {\it (1)\, $\Rightarrow$\, (3)\, $\Rightarrow$\,
(4)}.

Then is only necessary to proof that {\it (4)\, $\Rightarrow$\, (2)}
and we conclude the proof of Theorem B. This it follows directly from
the following result.
\end{proof}

\begin{pr}\label{JJ}
Let $f$ be a complex H\'enon map, dissipative with dominated splitting
in $J\sp*$. Then we have the following dichotomy:
\begin{itemize}
\item[{\it i.}] The set $\per$ is uniformly expanding at the
period.
\item[{\it ii.}] $J\sp*\cap J_0\neq\emptyset$.
\end{itemize}
\end{pr}
\begin{proof}
We assume that $J\sp*\cap J_0=\emptyset$, and that the set $\per$ is
not uniformly expanding at the period. In this case we can assume that
for every $n\geq1$, there exist a periodic point $p_n$ such that
\[||Df\sp{-k\pi(p_n)}|_{F(p_n)}||<\left(\frac{n-1}{n}\right)\sp{
k\pi(p_n)},\]
for every $k\geq1$. Thus we have
\begin{equation}\label{01111}
\log\left(\frac{n}{n-1}\right)>\frac{1}{k\pi(p_n)}\log\left(||Df\sp{
k\pi(p_n)}|_{F(p_n)}
||\right).
\end{equation}
Since that $\lambda\sp+(p_n)>0$, we can find $k_n$ great enough such
that
\begin{equation}\label{02}
\frac{1}{k_n\pi(p_n)}\log\left(||Df\sp{k_n\pi(p_n)}|_{F(p_n)}
||\right)>\frac{1}{n}.
\end{equation}
Now we define
\[\nu_n=\frac{1}{k_n\pi(p_n)}\sum\sp{k_n\pi(p_n)}_{j=1}\delta_{
f\sp{j}(p_n)},\]
be a sequence of $f$-invariant measures that, taking a subsequence if
necessary, we can assume that $\nu_n\rightarrow \nu$. It follows from
the inequalities (\ref{01111}) and (\ref{02}) that
\[\int\log||Df|_F||d\nu=\lim_{n\rightarrow\infty}
\int\log||Df|_F||d\nu_n=0,\]
that is a contradiction with $J_0=\emptyset$.
\end{proof}

Recently Christian Bonatti, Shaobo Gan and Dawei Yang, have proven an
more general case of the previous proposition and that contain this
(see \cite{betal}). In the work of Bonatti {\it Et al.},
an important hypothesis in the proof is that his compact invariant set
is a homoclinic class, and these is the case of $J\sp*$; but we don't
use this fact in the previous proof, however homoclinic class is a
hypothesis used in the proof of Forn\ae{}ss Theorem. We conclude this
subsection with the statement of Theorem of Bonatti Et al..

\begin{teo}[{\bf Bonatti-Gan-Yang}]Let $p$ be a hyperbolic periodic
point of a diffeomorphisms $f$ on a compact manifold $M$. Assume that
its homoclinic class $H(p)$ admits a (homogeneous) dominated splitting
$T_{H(p)}M=E\oplus F$ with $E$ contracting and
$\dim(E)=\textrm{ind}(p)$.

If $f$ is uniformly $F$-expanding at the period on the set of periodic
points $q$ homoclinically related to $p$, then $F$ is uniformly
expanding on $H(p)$.
\end{teo}

\subsection{Proof of Theorem \ref{JJJ}}
First one, we present the Theorem 2.1 due to Ma\~n\'e in \cite{Ma}.
Let $f$ be a diffeomorphisms of $C\sp1$ class in a Riemannian manifold
$M$ of any dimension, and $\Lambda$ be a compact invariant by $f$. A
dominated splitting $T\Lambda=E\oplus F$ is say {\it homogeneous} if
the dimension of the subspace $E(x)$ is constant for every $x\in
\Lambda$. We say that a compact neighborhood $U$ of $\Lambda$ is {\it
admissible} if the set $M(f,U)=\cap_{n\in\Z}f\sp n(U)$ has one and
exactly one homogeneous dominated splitting
$TM(f,U)=\widehat{E}\oplus\widehat{F}$ extending the splitting
$T\Lambda=E\oplus F$. It is known, that if $T\Lambda$ has a
homogeneous dominated splitting, then $\Lambda$ has an admissible
neighborhood $U$ (see \cite{HPS} for instance).

\begin{teo}
Let $\Lambda$ be a compact invariant set of $f\in\textrm{Diff}\sp1(M)$
such that $\Omega(f|_\Lambda)=\Lambda$, let $T_\Lambda M=E\oplus F$ be
a homogeneous dominated splitting such that $E$ is contracting and
suppose $c>0$ is such that the inequality \begin{equation}\label{03}
\liminf_{n\rightarrow\infty}\frac{1}{n}\sum_{j=1}\sp
n\log||(Df)\sp{-1}|_{F(f\sp j(x))}||<-c
\end{equation}
holds for a dense set of points $x\in \Lambda$. Then either $F$ is
expanding (and therefore $\Lambda$ is hyperbolic) or for every
admissible neighborhood $V$ of $\Lambda$ and every $0<\gamma<1$ there
exists a periodic point $p\in M(f,V)$ with arbitrarily large
period $N$ and satisfying
\[\gamma\sp N\leq\prod_{j=1}\sp N||(Df)\sp{-1}|_{\widehat{F}(f\sp
j(p))}||<1\]
where $\widehat{F}$ is given be the unique homogeneous dominated
splitting $TM(f,V)=\widehat{E}\oplus\widehat{F}$ that extend
$T\Lambda=E\oplus F$.
\end{teo}

% From the Lemma \ref{ro_pseudo_hyp} it follows that there exist
% $n_0$
% such that working with $f\sp{n_0}$ instead $f$, and taking
% $\lambda_0=\sqrt{b\lambda}$ and $\mu_0=\sqrt{b/\lambda}$ then
% uniformly in $J$ we have
% \begin{itemize}
% \item[(1)] $||Df\sp n|_{E(x)}||<\lambda_0\sp n$, for every
% $n\geq0$,
% \item[(2)] $||Df\sp n|_{F(x)}||>\mu_0\sp n$, for every $n\geq0$.
% \end{itemize}
% If $\lambda\leq b$, then $\mu_0>1$ and $J$ is hyperbolic. We will
% assume that $b<\lambda$.

In terms of the hypothesis of the Ma\~n\'e Theorem, is clear that are
satisfied for a dissipative H\'enon map: $f$ of $C\sp1$ class and
homogeneous dominated splitting. The inequality (\ref{03}) it is
satisfies with $c=\log(d)$ where $d$ is the degree of the map $f$. In
fact, in \cite{bls1} the authors proof that
\[
\lambda\sp-(p)=\log(b)-\lambda\sp+(p)\leq-\log(d)<0<\log(d)\leq\lambda
\sp+(p)\]
for every regular point for maximal entropy measure $\mu$, every
periodic saddle point is a regular point, and in \cite{bs1} is proved
that the saddle periodic point are dense in $J\sp*$.

Also we remark that any periodic point in $M(f,V)$ for some $V$ an
admissible neighborhood of $J\sp*$, is in fact an element of\,
$\per\subset J\sp*$.

\begin{proof}[{\bf Proof of Theorem \ref{JJJ}}] By the Ma\~n\'e
Theorem, if $J\sp*$ is not hyperbolic then, in par\-ti\-cu\-lar, for
every $n>0$ there exist a periodic point $p_n$ of period $N(n)\geq n$
such that
\[\log\left(\frac{n-1}{n}\right)\leq\frac{1}{N(n)}\log||(Df)\sp{-N(n)}
|_{{F}(p_n)} ||<0.\]
To end the proof, proceed in the same way as in the Proposition
\ref{JJ}.
\end{proof}

\subsection{Dominated splitting and partially hyperbolicity for
H\'enon maps}

In this subsection, we prove that dominated splitting H\'enon map, are
in fact partially hyperbolic. The complete statement is the following.

\begin{pr}\label{1aprop}
Let $f:\Ctwo\rightarrow\Ctwo$ a complex H\'enon map, with dominated
splitting $T_J\Ctwo=E\oplus F$ in $J$. Then
\begin{enumerate}
\item If $f$ is volume preserving, then $f$ is uniformly
hyperbolic in $J$.
\item If $f$ is dissipative, then $f$ is partially hyperbolic in
$J$ and the $E$ direction is a stable direction.
\end{enumerate}
\end{pr}

To the proof of this Proposition, we use a characterization of
dominated splitting that is proved in \cite{pancho}.

\begin{pr}
Let $f$ be a H\'enon map  with $b=|\det(Df)\, |$, and let
$T_J{\Ctwo}=E\oplus F$ be a splitting. The following statement are
equivalents:
\begin{enumerate}
\item The splitting $T_J\Ctwo=E\oplus F$ is dominated;
\item There exist $C>0$ and $0<\lambda<1$ such that:
\begin{itemize}
\item[{\it a)}] For every unitary vector $v\in F$ and
$n\geq1$
\[\frac{b\sp n}{||Df\sp nv||\sp2}\leq C\lambda\sp n,\]
\item[{\it b)}] For every unitary vector $v\in E$ and
$n\geq1$
\[\frac{b\sp{-n}}{||Df\sp{-n}v||\sp2}\leq C\lambda\sp n.\]
\end{itemize}
\end{enumerate}
\end{pr}

\begin{proof}[{\bf Proof of Proposition \ref{1aprop}.}]
From the previous Proposition, we can assume without loss of
generality that there exist $0<\lambda<1$ such that
\begin{itemize}
\item[(a)] $\displaystyle\frac{b\sp n}{||Df_x\sp
nu_x||\sp2}<\lambda\sp n$, for
every $n$ and $x\in J$,
\item[(b)]$\displaystyle\frac{b\sp{-n}}{||Df_x\sp{-n}v_x||\sp2}
<\lambda\sp n$, for every $n$ and $x\in J$,
\end{itemize}
where $u_x\in F(x)$ and $v_x\in E(x)$ are unitary vectors, and every
$x\in J$.

Replacing the previous inequality for the direction $E(x)$, it
follows that
\[||Df^{-n}_xv_x||^2>\left(\frac{1}{b\lambda}\right)^n.\]
Replacing the inverse function of $Df^{-n}$ in the previous
inequality, and taking $\lambda_0=\sqrt{b\lambda}$, we obtain that
\[||Df^n_xv_x||\leq \lambda_0^n\, \, \, \, \Longrightarrow\, \, \, \,
||Df^n_x|_{E(x)}||\leq\lambda_0^n.\]
Similarly for the direction $F(x)$, let $u_x$ a unitary vector in this
direction we obtain
\[||Df^{n}_xu_x||^2>\left(\frac{b}{\lambda}\right)^n,\]and taking
$\mu_0=\sqrt{b/\lambda}$ it follows that
\[||Df^n|_{F_x}||\geq \mu_0^n.\]
Thus we have\[\lambda^2<1\, \Longleftrightarrow\,
b\lambda<\frac{b}{\lambda}\, \Longleftrightarrow\,
\lambda_0<\mu_0.\]This prove that any complex H\'enon map with
dominated splitting in $J$ is partially hyperbolic, so this are
equivalent notions in this context.

To prove the item 1, i.e. the volume preserving case $b=1$, is only
necessary to observe that
\[\lambda_0=\sqrt{\lambda}<1<\sqrt{1/\lambda}=\mu_0,\]
and for the item 2, i.e. the dissipative case $b<1$, we have that
so $\lambda_0=\sqrt{\lambda b}<1$, then $E$ is a stable direction, as
is desired.
\end{proof}

\subsection{Weak forward expansivity in $J\sp*$}
Periodic saddle point in $J\sp*$ have unstable manifold, and this can
by characterized as the set of point in which the function $f$ has
asymptotically expansiveness, and the constant of expansivity is
related with the rate of expansion of the derivative in the unstable
direction. This implies, that the map is forward expansivity along the
orbit of a periodic point. The problem appear because the constant of
expansiveness in the unstable direction is not uniform in the set of
periodic point, so the forward expansivity is not uniform in the set
of periodic saddle point.

Notwithstanding the above fact, in each unstable manifold of a
periodic saddle point, there are many point (an open set in each
unstable manifold) that goes to infinity by positives iterates of $f$.
Then we can say that in this points, we have an uniform forward
expansivity. This property over periodic saddle points, for a
dissipative H\'enon map with dominated splitting, can be recovered
over each point in the support of the maximal entropy measure $J\sp*$.
This is stated in the following proposition.

For notations and definition of the Julia set $J$, the support of
the measure of maximal entropy $J\sp*$ and the set $U\sp+$, see
\cite{bs1}.

\begin{pr}\label{hol6}
Let $f$ be dissipative H\'enon map, with dominated splitting in
$J\sp*$. Then for every $x\in J\sp *$, holds that
$\www{cu}{x}{loc}\cap U\sp+\neq\emptyset$.
\end{pr}

\begin{proof}
The statement of the Proposition is true for saddle periodic points.
In fact, for a saddle periodic point $p$, we have that $\ww{u}{p}$ is
a copy of $\C$ (see \cite{bs1}), and is dense in $J\sp-$ (see
\cite{bs2}). Also, we have that $J\sp-\cap
U\sp +\neq\emptyset$ because $J\sp-=\partial K\sp-$ and $K\sp-\cap
U\sp+\neq\emptyset$.

Thus, to proof this Proposition, we assert that the stable manifold of
$p$ intersect any local center-unstable disk. This it follows from the
fact that $J\sp*=H(p)$ is a homoclinic class of any periodic saddle
point, and that there is a uniformly contractive sub-bundle, i.e., the
direction $E$ (see Proposition \ref{1aprop}).

Let $p_k$ be a sequence of periodic saddle points, $p_k\rightarrow
x\in J\sp*$. From the continuity of the splitting, it follows that for
$k$ great enough $\www{s}{p_k}{loc}\cap \www{cu}{x}{loc}\neq
\emptyset$. Since that $p_k\in H(p)$, each $\www{s}{p_k}{loc}$ is
approximated by disc contained in $\ww{s}{p}$. More precisely, there
exist a disc $D_k\subset\ww{s}{p}$ such that
$\d_1(D_k,\www{s}{p_k}{loc})<1/k$, where $\d_1$ is the metric of the
$C\sp1$ topology.
% , and that each $D_k$ intersect $\www{u}{p_k}{loc}$.
It follows that for $k$ great enough,
$D_k\cap\www{cu}{x}{loc}\neq\emptyset$ thus $\ww{s}{p}$ intersect to
$\www{cu}{x}{loc}$.

Now since that $\ww{u}{p}\cap U\sp +\neq\emptyset$ and $U\sp+$ is
open, backward iterates of $U\sp+$ accumulates on any compact part of
$W\sp s(p)$, and imply that backward iterates of $U\sp +$ intersects
$\www{cu}{x}{loc}$.
\end{proof}

\end{document}